\RequirePackage{etex}
\RequirePackage{easymat}

\documentclass[12pt, a4paper]{elsarticle}
\usepackage{amsthm,amsmath,amssymb,extsizes,color}
\usepackage[active]{srcltx}
\usepackage{MnSymbol}
\usepackage{mathrsfs}
\DeclareMathOperator{\Ind}{Ind}
\DeclareMathOperator{\codim}{codim}

\newcommand{\ddd}{
\text{\begin{picture}(12,8)
\put(-2,-4){$\cdot$}
\put(3,0){$\cdot$}
\put(8,4){$\cdot$}
\end{picture}}}

\renewcommand{\le}{\leqslant}
\renewcommand{\ge}{\geqslant}

\newtheorem{theorem}{Theorem}[section]
\newtheorem{lemma}{Lemma}[section]

\theoremstyle{definition}
\newtheorem{definition}{Definition}[section]

\theoremstyle{remark}
\newtheorem{remark}{Remark}[section]
\newtheorem{corollary}{Corollary}[section]

\newcommand{\hide}[1]{}

\begin{document}

\title{Miniversal deformations of pairs of symmetric matrices under congruence}




\date{}

\author[um]{Andrii Dmytryshyn}
\ead{andrii@cs.umu.se}

\address[um]{Department of Computing Science, Ume\aa \, University, SE-901 87 Ume\aa, Sweden.
\newline 
\vskip0.2cm
\small{\rm Dedicated to Vladimir V. Sergeichuk on the occasion of his 70th birthday.}
\vskip-0.2cm}

\begin{abstract}
For each pair of complex symmetric matrices $(A,B)$ we provide a normal form with a minimal number of independent parameters, to which all pairs of complex symmetric matrices $(\widetilde{A},\widetilde{B})$, close to $(A,B)$ can be reduced by congruence transformation that smoothly depends on the entries of $\widetilde{A}$ and $\widetilde{B}$. Such a normal form is called a miniversal deformation of $(A,B)$ under congruence. A number of independent parameters in the miniversal deformation of a symmetric matrix pencil is equal to the codimension of the congruence orbit of this symmetric matrix pencil and is computed too. We also provide an upper bound on the distance from $(A,B)$ to its miniversal deformation.  
\end{abstract}


\maketitle

\section{Introduction}
Finding a normal form to which all matrices $\widetilde{A}$, close to a given matrix $A$, can be reduced by certain transformation which smoothly depends on the entries of the matrix $\widetilde{A}$ is a challenging problem. In 1971 V.I. Arnold introduced such a normal form for matrices under similarity \cite{Arno71}, see also \cite[\S\,30B]{Arno97}, and called it a (mini)versal deformation. The prefix ``mini-'' is added if the number of independent parameters in the normal form is minimal. 
%
%
Now the notion of miniversal deformations has been extended to general \cite{EdEK97,GaSe99} and structured \cite{Dmyt16,DmFS12,DmFS14} matrix pencils, matrices of bilinear \cite{DmFS12} and sesquilinear \cite{DmFS14} forms, as well as to matrices under similarity over various fields \cite{BoSS17,Gali72,GaSe99}.

In this paper, we derive a miniversal deformation of a pair $(A,B)$ of symmetric matrices ($A= A^T$ and $B = B^T$) under congruence; that is, a normal form with the minimal number of independent parameters, to which all pairs $(\tilde{A},\tilde{B})$ of symmetric matrices close to $(A,B)$ can be reduced by {\it congruence transformations} smoothly depending on the entries of $\tilde{A}$ and $\tilde{B}$. Recall that a pair of $n \times n$ symmetric matrices $(A,B)$ is called {\it congruent} to $(C,D)$ if and only if there is a nonsingular matrix $S$ such that $S^TAS = C$ and $S^TBS = D$. The set of pairs of matrices congruent to a pair of symmetric matrices $(A,B)$ is called a {\it congruence orbit} of $(A,B)$. 
The codimension of the congruence orbit of a pair of symmetric matrices is equal to the number of independent parameters in the miniversal deformation of this pair and is computed in this paper too. 
We also bound the distance from the deformations to unperturbed pairs of matrices in terms of the norm of the perturbations. 
When talking about the previous results, we will sometimes use the term ``matrix pencil'' instead of ``pair of matrices''  (in the context of this paper these terms are equivalent). 



Symmetric matrix pencils appear in a wide range of applications, including motion or vibration of structural systems \cite{Parl91,TiMe01}, viscous damping \cite{Dumo07}, network theory \cite{BeHK15}. Often symmetric matrix pencils appear as a result of symmetric linearizations for symmetric matrix polynomials \cite{AnVo04,HMMT07}. Many of these applications require computing eigenstructures of matrix pencils, for example, via a structured staircase form for symmetric matrix pencils \cite{BrMe07} as well as understanding the behaviour of these eigenstructures under low rank \cite{MeMW17} and general perturbations, and that is where our miniversal deformations may be useful \cite{Dmyt16,DFKK15,FuKS14}.  
%
Moreover, based on the versal deformation theory, a constructive approach to determine the geometry of the singularities (orientation in space, magnitudes of angles, etc.) by constructing tangential cones to the stability domain is developed in \cite{MaSe97,MaSe99}. Some applications of miniversal deformations in control and stability theories can be found in \cite{EdEK97,KaJJ12,Mail00}. In particular, miniversal deformations of symmetric matrix pencils can help us to construct their stratifications, i.e. closure hierarchies of orbits and bundles, see the examples in \cite{Dmyt16,DFKK15,FuKS14}. These stratifications are illustrated by the graphs showing all canonical forms that the symmetric matrix pencils may have in arbitrarily small neighbourhoods of a given symmetric matrix pencil.
For example, the stratifications show how a Jordan-like block can split into two Jordan-like blocks associated with two different eigenvalues. The stratification graphs are known for matrices \cite{EdEK99}, general matrix pencils \cite{EdEK99}, matrix pencils associated with state-space systems \cite{DmJK17}, matrix polynomials \cite{DJKV15,JoKV13}, as well as for the skew-symmetric matrix pencils \cite{DmKa14} and polynomials \cite{Dmyt17}. 
%
Nevertheless the stratification theory for symmetric matrix pencils remains to be an open and challenging problem and this paper can be seen as a step towards a better understanding of small perturbations of symmetric matrix pencils and thus towards a development of the stratification theory. 

This paper and the paper \cite{Dmyt16} are directed towards the same audience. To facilitate the reading and the use of results from \cite{Dmyt16}, we keep their style, structure, notation as similar as possible, and organize the rest of this paper as follows.
In Section $2,$ we start by recalling some preliminary information needed to present the miniversal deformations of symmetric matrix pencils. We also give an upper bound on the distance between a symmetric matrix pencil and its miniversal deformation as well as compute the codimensions of the congruence orbits of symmetric matrix pencils. (The Matlab functions for computing these codimensions were developed \cite{DmJK13} and became a part of the Matrix Canonical Structure (MCS) Toolbox \cite{Joha06}.)
In Section $3.1$ we present a method for constructing the miniversal deformations. In the remaining parts of Section $3$ we derive the deformations step by step, namely, for the diagonal blocks in Section $3.2$, for the off-diagonal blocks that correspond to the canonical summands of the same type in Section $3.3$, and finally, for the off-diagonal blocks that correspond to the canonical summands of different types in Section $3.4$.

In this paper all matrices are considered over the field of complex numbers. To refer to a matrix pair, we use calligraphic letters, e.g., ${\cal A }$ or ${\cal D}$.

\section{Miniversal deformations of pairs of symmetric matrices}
We start this section by recalling the canonical form of pairs of symmetric matrices under congruence given in \cite{Thom91}, then we present some preliminaries on miniversal deformations followed by our main theorems. The miniversal
deformations derived in Theorem \ref{teo2} will be proven in Section \ref{sec3}.    


For each $n=1,2,\dots $, define the $n\times n$
matrices
\begin{equation*}\label{1aa}
\Lambda_n(\lambda):=
\begin{bmatrix}
&&&\lambda\\
&&\lambda&1\\
&\ddd&\ddd&\\
\lambda&1&&\\
\end{bmatrix},\qquad
\Delta_n:=\begin{bmatrix}&&&1\\
&&1&\\
&\ddd&&\\
1&&&\\
\end{bmatrix},
\end{equation*}
where $\lambda \in \mathbb C$, and the $n \times (n+1)$ matrices
\begin{equation*}
F_n :=
\begin{bmatrix}
1&0&&\\
&\ddots&\ddots&\\
&&1&0\\
\end{bmatrix}, \qquad
G_n :=
\begin{bmatrix}
0&1&&\\
&\ddots&\ddots&\\
&&0&1\\
\end{bmatrix}.
\end{equation*}
All non-specified entries of the matrices $\Lambda_n(\lambda),\Delta_n,F_n,$ and $G_n$ are zeros. 

\begin{lemma}[\cite{Thom91}] \label{lkh}
Every pair of
symmetric complex
matrices is congruent
to a direct sum 
\begin{equation}\label{kus}
(A,B)_{\rm can}=
\bigoplus_{i=1}^{a} \mathcal  H_{h_i}(\lambda_i)
 \oplus
\bigoplus_{j=1}^{b} \mathcal K_{k_j}
 \oplus
\bigoplus_{r=1}^{c} \mathcal L_{l_r},
\end{equation}
where 
\begin{align}
\label{can}
\mathcal {\cal H}_n(\lambda)&:= (\Delta_n ,
\Lambda_n(\lambda) ),\quad \lambda \in\mathbb C,\\
\label{can2}
\mathcal {\cal K}_n&:= (\Lambda_n(0),\Delta_n ),\\
\label{can3}
\mathcal {\cal L}_n&:= \left(
\begin{bmatrix}0&F_n^T\\
F_n &0
\end{bmatrix},
\begin{bmatrix}0&G^T_n\\
G_n&0
\end{bmatrix}
 \right).
\end{align}
The sum \eqref{kus} is determined uniquely up
to permutation of summands.
\end{lemma}
\hide{
\noindent Thus, each pair of symmetric matrices is congruent to a direct sum
\begin{equation}\label{kus}
(A,B)_{\rm can}=
\bigoplus_{i=1}^{a} \mathcal  H_{h_i}(\lambda_i)
 \oplus
\bigoplus_{j=1}^{b} \mathcal K_{k_j}
 \oplus
\bigoplus_{r=1}^{c} \mathcal L_{l_r},
\end{equation}
consisting of direct summands of three
types {of pairs}.}



We extend Arnold's concept of miniversal deformations to pairs of symmetric matrices with respect to congruence in the same manner as it was done for pairs of skew-symmetric matrices \cite{Dmyt16}. Similarly miniversal deformations has been defined for matrix pencils \cite{EdEK97,GaSe99}, as well as for matrices of bilinear \cite{DmFS12} and sesquilinear \cite{DmFS14} forms. 

A \emph{deformation}
of a pair of symmetric $\hat{n}\times \hat{n}$ matrices $(A,B)$ is
a holomorphic mapping
${\cal A}(\vec\delta)$, where $\vec\delta=(\delta_1,\dots, \delta_k)$, from a
neighborhood
$\Omega \subset
\mathbb C^k$ of $\vec
0=(0,\dots,0)$ to the
space of pairs of
symmetric $\hat{n}\times \hat{n}$ matrices
such that ${\cal A}(\vec 0)=(A,B)$. In this paper we consider only symmetric deformations, i.e. we preserve the symmetric structure of matrix pairs. Thus, with no risk of confusion, we write ``deformation'' but not ``symmetric deformation''.    

\begin{definition}\label{d}
A deformation ${\cal
A}(\delta_1,\dots,\delta_k)$
of a pair of symmetric matrices 
$(A,B)$ is called
\emph{versal} if for every
deformation ${\cal
B}(\sigma_1,\dots,\sigma_l)$
of $(A,B)$ we have  
\begin{equation*}\label{kft}
{\cal B}(\sigma_1,\dots,\sigma_l)=
I(\sigma_1,\dots,\sigma_l)^{T}
{\cal
A}(\varphi_1(\vec\sigma),\dots,
\varphi_k(\vec\sigma))
I(\sigma_1,\dots,\sigma_l),
\end{equation*}
where $I(\sigma_1,\dots,\sigma_l)$ is a deformation of the identity matrix, 
and all
$\varphi_i(\vec\sigma)$
are convergent in a
neighborhood of $\vec
0$ power series such
that $\varphi_i(\vec
0)=0$. A versal
deformation ${\cal
A}(\delta_1,\dots,\delta_k)$
of $(A,B)$  is called
\emph{miniversal} if
there is no versal
deformation having
less than $k$
parameters.
\end{definition}

Define a \emph{$(0,*)$
matrix} to be a
matrix whose entries
are $0$ and $*$. 
A pair of symmetric matrices
\emph{is of the form}
$\cal D$, where $\cal D$ is a pair of $(0,*)$ symmetric matrices, if it can be
obtained from $\cal D$
by replacing the stars
with complex numbers, respecting the symmetry.
By ${\cal D}({\mathbb C})$ we denote the
space of all pairs of symmetric matrices 
of the form
$\cal D$, and by
${\cal D}(\vec
{\varepsilon})$ we denote the
pair of parametric symmetric 
matrices obtained from
${\cal D}$ by
replacing the {$(i,j)$th and $(j,i)$th
stars, $i\leq j$, with the parameter
${\varepsilon}_{ij}$ 
in the first matrix, and the $(i',j')$th and $(j',i')$th
stars, $i' \leq j'$, with the parameter ${\varepsilon}^{'}_{i'j'}$ 
in the second matrix.} That is to say  
\begin{equation} \label{a2z}
{\cal D}(\vec
{\varepsilon}):=
\left(\sum_{(i,j) \in  
\Ind_1({\cal D})}
\varepsilon_{ij}E_{ij},
\sum_{(i',j')\in \Ind_2({\cal D})}
{\varepsilon}^{'}_{i'j'}E_{i'j'}\right),
\end{equation}
\begin{equation} \label{spacedc}
{\cal D}(\mathbb C):= 
\left\{
{\cal D}(\vec{\varepsilon}) \ | \ \vec{\varepsilon} \in \mathbb C^k
\right\} = 
\left\{
\Big(\bigplus_{(i,j)\in 
\Ind_1({\cal D})}
{\mathbb C}
E_{ij}, 
\bigplus_{(i',j')\in
\Ind_2({\cal D})}
{\mathbb C}
E_{i'j'}\Big)
\right\},
\end{equation}
where each $E_{ij}$ is the matrix
whose $(i,j)$th and $(j,i)$th entries 
are $1$, and the other entries 
are zero, ``$\bigplus$'' denotes 
the entrywise sum of matrices,  and
$\Ind_1({\cal
D}),\Ind_2({\cal
D})\subseteq
\{1,\dots, \hat{n}\}\times
\{1,\dots, \hat{n}\},$
are the sets of
indices of the stars
in the upper-triangular parts (including the main diagonals) of the first and
the second matrices of the pair ${\cal D}$.
A miniversal deformation
of $(A,B)$ is
\emph{simplest} if it
has the form
$(A,B)+{\cal D}(\vec
{\varepsilon})$, where
$\cal D$ is a pair of $(0,*)$
matrices, see also \cite{Dmyt16,DmFS12,DmFS14,GaSe99}.

{ In other words,} for all pairs of $\hat{n} \times \hat{n}$ symmetric matrices
$(A+E,B+E')$ that are close to a given
pair of symmetric
matrices $(A,B)$,
we derive the normal form ${\cal A}(E,E')$ with respect to the congruence
transformation
\begin{equation} \label{trans}
(A+E,B+E') \mapsto S(E,E')^T(A+E,B+E') S(E,E') =: {\cal A}(E,E'),
\end{equation}
in which $S(E,E')$ is nonsingular and holomorphic at 0 (i.e. its
entries are power series in the entries of $E$ and $E'$
that are convergent in a neighborhood of 0). 

We define ${\cal A}(0,0)$ 
to be equal to the
congruence canonical form $(A,B)_{\text{\rm can}}$, see \eqref{kus}, of $(A,B)$.
Then
\begin{equation}
\label{ksy} {\cal
A}(E,E')=
(A,B)_{\text{\rm
can}}+{\cal D}(E,E'),
\end{equation}
where  ${\cal
D}(E,E')$ (${\cal D}(E,E')={\cal D}(\vec{\varepsilon}) $ for some $ \vec{\varepsilon} \in \mathbb C^k$) is a pair of symmetric matrices that is
holomorphic at $0$ and
${\cal D}(0,0)=(0,0)$.
In Theorem \ref{teo2} we
present ${\cal
D}(E,E')$ with the
minimal number of
nonzero entries that is attainable using the congruence transformation \eqref{trans}.

Define the following $(0,*)$
matrices, where each star
denotes a holomorphic at zero function of
the entries of $E$ and
$E'$:

$\bullet$ $0_{mn}$ is
the $m \times n$ zero
matrix;

$\bullet$
$0_{mn \ast}$
is the $m \times n$
matrix
$
\begin{bmatrix}
&  && 0\\
&0_{m-1,n-1}&& \vdots\\
&  && 0\\
0&\ldots&0&*
\end{bmatrix};
$

$\bullet$
$0_{mn}^{\leftarrow}$ (resp. $0_{mn}^{\rightarrow}$) 
is the
$m \times n$ matrix
$\label{bjhf}
\begin{bmatrix}
  * &  \\
   \vdots & 0_{m,n-1} \\
  * &
\end{bmatrix}$ $\left( \text{ resp. } 
\begin{bmatrix}
  & *  \\
  0_{m,n-1} & \vdots \\
  & *
\end{bmatrix}
\right)$;

\hide{
$\bullet$
$0_{mn}^{\rightarrow}$ is, respectively, the
$m \times n$ matrix  $
\begin{bmatrix}
  & *  \\
  0_{m,n-1} & \vdots \\
  & *
\end{bmatrix};
$}

$\bullet$
$0_{mn}^{\nwarrow}$
is the $m \times n$
matrix
\begin{equation}\label{mgde}
\begin{bmatrix}
\begin{matrix}
*\\ *\\\vdots\\ *
\end{matrix}
&0_{m,n-1}
\end{bmatrix}\quad
 \text{if $m\le n$, and }
\begin{bmatrix}
  *\ *\ \dots\ *
 \\[2mm]
   0_{m-1,n}
\\[2mm]
\end{bmatrix}\quad
\text{if $m\ge n$},
 \end{equation}
 
if $m=n$, then we can
choose any of the
matrices defined in
\eqref{mgde};

$\bullet$
$0_{mn}^{\righthalfcap}$
is the $m \times n$
matrix
$
\begin{bmatrix}
* & \ldots & *\\
&0_{m-1,n-1}& \vdots\\
&&*
\end{bmatrix}\quad
 \text{or}\quad
 \begin{bmatrix}
* & & \\
\vdots&0_{m-1,n-1}& \\
*&\ldots & *
\end{bmatrix};
$

\hide{
$\bullet$
$0_{nn}^{\nwarrow\!\!\!\! \text{\raisebox{1.5pt}{$\nwarrow$}} \!\!\!\! \text{\raisebox{3pt}{$\nwarrow$}}}$
is the $n \times n$
matrix (here and after unspecified entries are zeros)
$$
\left[\begin{array}{c|c}
\begin{matrix} *&*& \\
*& \ddots& \ddots \\
&\ddots&*
\end{matrix}& \begin{matrix}&&\\
&&\\
*&&
\end{matrix}\\
\hline
\begin{matrix}&&&*\\
&&&\\
&&&
\end{matrix}&\\
\end{array}\right]  \quad
 \text{when $n$ is even and}
\left[\begin{array}{c|c|ccc}
\begin{matrix} *&*& \\
*& \ddots& \ddots \\
&\ddots&*
\end{matrix}& \begin{matrix}\\
\\
*
\end{matrix}&&&\\
\hline
\begin{matrix}&&&*
\end{matrix}&*&&&\\
\hline
&&&&\\
&&&&\\
&&&&\\
\end{array}\right] \quad
 \text{when $n$ is odd};
$$
}

$\bullet$
$0_{nn}^{\nwarrow\!\!\!\! \text{\raisebox{1.5pt}{$\nwarrow$}} \!\!\!\! \text{\raisebox{3pt}{$\nwarrow$}}}$
is the $n \times n$ matrix { 
{\footnotesize
$
\setlength{\arraycolsep}{3pt}\begin{bmatrix}
*& *&  & &&&&\cdot\\
*& & \ddots    & & &&\cdot&\\
& \ddots & \ddots& & &\cdot&&\\
&& & *&*&&&\\
&&&*&0&0&&\\
&&\cdot&&0&0&\ddots&\\
&\cdot&&&&\ddots&\ddots&0\\
\cdot&&&&&&0&0\\
\end{bmatrix}
$} or {\footnotesize
$
\setlength{\arraycolsep}{3pt}\begin{bmatrix}
*& *&  & &&&\cdot\\
*& \ddots & \ddots     & &&\cdot&\\
& \ddots&  *   & * &\cdot&&\\
&&*&*&0&&\\
&&\cdot&0&0&\ddots&\\
&\cdot&&&\ddots&\ddots&0\\
\cdot&&&&&0&0\\
\end{bmatrix}
$}}

for $n$ being even or, respectively, odd; $\lceil n/2 \rceil$ stars on the main diagonal 

and $\lfloor n/2 \rfloor$ stars on each of the sub- and overdiagonals;

$\bullet$
$0_{nn}^{\nwsearrow\!\!\!\! \text{\raisebox{1.5pt}{$\nwsearrow$}} \!\!\!\! \text{\raisebox{3pt}{$\nwsearrow$}}}$
is the $n \times n$
matrix
$
 \begin{bmatrix}
*& *&  & \\
*& *& \ddots     & \\
& \ddots& \ddots    & * \\
&&*&*
\end{bmatrix};
$

$\bullet$ $0^{\boxminus}_{mn}$ with $m < n$
is the $m \times n$
matrix
\begin{equation*}\label{hui}
\begin{bmatrix}
\begin{matrix}
0&\dots& 0\\
\vdots&& \vdots
\end{matrix} &0_{n-1,n-m+1}\\
\begin{matrix}
0& \dots&0\end{matrix}&
\begin{matrix}
 *\ \dots\ * \ 0\
\end{matrix}
\end{bmatrix}\qquad(\text{$n-m$
stars}) 
\end{equation*}

if $m \geq n$ then $0^{\boxminus}_{mn}=0$.

\noindent If there is no risk of confusion, we will omit the indices $m$ and $n$.

Consider a canonical pair of
symmetric matrices under congruence
\begin{equation}\label{gto}
(A,B)_{\text{\rm \rm
can}}={\mathcal X}_1\oplus\dots
\oplus {\mathcal X}_t, 
\end{equation}
where
${\mathcal X}_1,\dots,{\mathcal X}_t$ are
pairs of the form
\eqref{can}--\eqref{can3}, and let 
${\cal D}(E,E')$ be a
pair of symmetric matrices, defined in \eqref{ksy}, whose
matrices are
partitioned into
blocks conformally to
the decomposition
\eqref{gto}, i.e.
\begin{equation}\label{grsd}
{\cal D}(E,E')= {\cal D} = \left(
\begin{bmatrix}
D_{11}&\dots&
D_{1t}
 \\
\vdots&\ddots&\vdots\\
D_{t1}&\dots&
D_{tt}
\end{bmatrix},
\begin{bmatrix}
D'_{11}&\dots&
D'_{1t}
 \\
\vdots&\ddots&\vdots\\
D'_{t1}&\dots&
D'_{tt}
\end{bmatrix}
\right), \ (D_{ji},D'_{ji})=(D_{ij}^T,D_{ij}^{'T}). 
\end{equation}
Define
\begin{equation}\label{lhsd}
{\cal D}(\mathcal X_i):=(D_{ii},D'_{ii}) \quad  \text{ and } \quad 
 {\cal
D}(\mathcal X_i, \mathcal X_j):=(
D_{ij},D'_{ij}), \ i<j.
\end{equation}

{ It is sufficient to construct miniversal deformations for all canonical pairs of matrices; i.e., for all direct sums} of the pairs \eqref{can}--\eqref{can3}, since each pair of symmetric matrices is congruent to its canonical form, see Lemma \ref{lkh}. 

\begin{theorem}\label{teo2}
Let $(A,B)_{\text{\rm \rm can}}$
be a canonical pair of symmetric complex
matrices 
\eqref{kus}. 
A simplest
miniversal deformation of $(A,B)_{\text{\rm can}}$
can be taken in the
form $(A,B)_{\text{\rm
can}} +{\cal D}$ in which ${\cal D}$ is a
pair of $(0,*)$ matrices (the stars denote
independent parameters, up to symmetry, see also Remark \ref{indep}) whose
matrices are partitioned into blocks conformally to
the decomposition of $(A,B)_{\text{\rm can}}$, see \eqref{grsd}, and 
the blocks of ${\cal D}$ are
defined, in the notation \eqref{lhsd}, as follows:  

{\rm(i)} The
diagonal blocks of
${\cal D}$ are defined
by
\begin{align}\label{Hdef}
{\cal
D}({\cal H}_n(\lambda))&=
\left( 0, 0^{\nwarrow\!\!\!\! \text{\raisebox{1.5pt}{$\nwarrow$}} \!\!\!\! \text{\raisebox{3pt}{$\nwarrow$}}} \right),\\
\label{Kdef}
{\cal D} ({\cal K}_n)&=\left(
0^{\nwarrow\!\!\!\! \text{\raisebox{1.5pt}{$\nwarrow$}} \!\!\!\! \text{\raisebox{3pt}{$\nwarrow$}}}, 0 \right),\\
\label{Ldef}
{\cal D}({\cal L}_n)&=\left(
 \begin{bmatrix}
0_{*}&0
 \\ 0&0
\end{bmatrix} ,
\begin{bmatrix}
0^{\nwsearrow\!\!\!\! \text{\raisebox{1.5pt}{$\nwsearrow$}} \!\!\!\! \text{\raisebox{3pt}{$\nwsearrow$}}}&0
 \\ 0&0
\end{bmatrix}
\right).
\end{align}

{\rm(ii)} The
off-diagonal blocks of
${\cal D}$ whose
horizontal and
vertical strips
contain summands of
$(A,B)_{\text{\rm
can}}$ of the same
type are defined by
\begin{align}\label{lsiu1}
{\cal D}
({\cal H}_n(\lambda),\,
{\cal H}_m(\mu))
       &=
  \begin{cases}
(0,\:0) &\text{if
$\lambda\ne\mu$,} \\
    \left(0,0^{\nwarrow}\right)
 &\text{if $\lambda=\mu$,}
  \end{cases}\\
\label{lsiu2}
{\cal D} ({\cal K}_n,{\cal K}_m)&=
  \left( 0^{\nwarrow} , 0 \right),\\
\label{lsiu3}
{\cal D}
({\cal L}_n,{\cal L}_m)&=\left(
 \begin{bmatrix}
0_{\ast}&0
 \\ 0&0
\end{bmatrix} ,
 \begin{bmatrix}
0^{\righthalfcap}&0^{\boxminus}_{n+1,m}
 \\ 0^{\boxminus T}_{m+1,n}&0
\end{bmatrix}
\right).
\end{align}
{\rm(iii)} The
off-diagonal blocks of
${\cal D}$ whose
horizontal and
vertical strips
contain summands of
$(A,B)_{\text{\rm
can}}$ of different
types are defined by:
\begin{align}\label{kut}
{\cal D}
({\cal H}_n(\lambda),{\cal K}_m)&=(0,0), \\
\label{hnlm}
{\cal D}
({\cal H}_n(\lambda),{\cal L}_m)&=\left(
0,0^{\leftarrow}\right),\\
\label{ktlm}
 {\cal D}
({\cal K}_n,{\cal L}_m)&=\left(
 \begin{bmatrix}
0^{\rightarrow}&0
\end{bmatrix} ,
0
\right).
\end{align}
\end{theorem}

\begin{remark}[Independency of parameters] \label{indep}
All parameters that are placed instead of the stars in the upper triangular parts, including the main diagonals, of matrices of ${\cal D}$ are independent and the strictly lower triangular parts of matrices of ${\cal D}$ are defined by the symmetry. For example, it means that parametric matrix pairs obtained from $(D_{ij},  D'_{ij})$ and $(D_{i'j'}, D'_{i'j'})$ have { dependent (actually, equal) parameter entries} if and only if $i'=j$ and $j'=i$.
\end{remark}


Now we explain how the matrix pair $\cal D$ defined in \eqref{grsd} is constructed in Section
\ref{sec3}. The vector space
\begin{equation*}\label{msi}
T_{(A,B)}:=\{C^T(A,B)_{\text{\rm
can}}
+(A,B)_{\text{\rm
can}}C\,|\,
C\in{\mathbb
C}^{\hat{n}\times \hat{n}}\}
\end{equation*}
is the tangent space
to the congruence
class of
$(A,B)_{\text{\rm
can}}$ at the point
$(A,B)_{\text{\rm
can}}$. 
\hide{
since
\[
(I+\varepsilon
C)^T(A,B)_{\text{\rm
can}} (I+\varepsilon
C) =(A,B)_{\text{\rm
can}}+ \varepsilon(C^T
(A,B)_{\text{\rm
can}}+
(A,B)_{\text{\rm
can}}C)
 \]
 \[+\varepsilon^2C^T
(A,B)_{\text{\rm
can}}C
\]
for all $\hat{n} \times \hat{n}$
matrices $C$ and each
$\varepsilon\in\mathbb
C$.} 
Then $\cal D$
satisfies the
following condition:
\begin{equation}\label{jyr}
{\mathbb C}^{\,\hat{n} \times
\hat{n} }_{s}\times{\mathbb
C}^{\,\hat{n} \times \hat{n}}_{s}=T_{(A,B)}
+ {\cal
D}({\mathbb C})
\end{equation}
where ${\mathbb
C}^{\,\hat{n} \times \hat{n}}_{s}$
is the space of all $\hat{n}
\times \hat{n}$
symmetric matrices and 
${\cal D}({\mathbb C})$, defined in \eqref{spacedc}, is
the vector space of
all matrix pairs
obtained from $\cal D$
by replacing its stars
by complex numbers.
Thus, the number of
stars in the uppertriangular part of $\cal D$ (including the main diagonal) is
equal to the
codimension of the
congruence class of
$(A,B)_{\text{\rm
can}}$. Recall that the codimension of congruence orbit of $(A,B)$ is defined as the dimension of the normal space $N_{(A,B)}$ at the point $(A,B)$ and $N_{(A,B)}$ is the orthogonal complement to the tangent space $T_{(A,B)}$ in ${\mathbb C}^{\,\hat{n} \times\hat{n} }_{s}\times{\mathbb C}^{\,\hat{n} \times \hat{n}}_{s}$. 

Following \cite{Dmyt16} and using the norm of the original perturbations, we bound the distance from the miniversal deformations of a symmetric matrix pair to this matrix pair. Notably, this distance can be made arbitrarily small by decreasing the size of the allowed perturbations.  

\hide{
Recall that for matrices $Y$ and
$Z$ and $\nu, \omega \in\mathbb
C$ the following inequalities hold (e.g., see \cite[Section
5.6]{HoJo85})
\begin{equation}\label{lk}
\|\nu Y+ \omega Z\|\le
|\nu |\,\|Y\|+|  \omega |\,\|Z\| \quad \text{and} \quad
\|YZ\|\le \|Y\|\,\|Z\|.
\end{equation}
Let $(A,B)\in({\mathbb
C}^{\,\hat{n}\times \hat{n}}_c,{\mathbb
C}^{\,\hat{n}\times \hat{n}}_c)$ and $\alpha:=\|A\| , \beta:=\|B\|$. 
}
By
\eqref{jyr}, for each pair of symmetric ${\hat{n}}$-by-${\hat{n}}$ matrices
$(E_{ij},0)$ and $(0,E_{i'j'})$, $1 \le i,j,i',j' \le {\hat{n}}$ there exist $X_{ij},X'_{i'j'}\in\mathbb
C^{\hat{n}\times \hat{n}}$ such
that
\begin{equation*}
\begin{aligned}\label{8}
(E_{ij},0)&+X_{ij}^T(A,B)
+(A,B)X_{ij} \in {\cal
D}({\mathbb C}),\\
(0,E_{i'j'})&+X_{i'j'}^{'T}(A,B)
+(A,B)X'_{i'j'} \in {\cal
D}({\mathbb C}),
\end{aligned}
\end{equation*}
where ${\cal D}({\mathbb C})$ is
defined in \eqref{spacedc}.
If $(i,j)\in \Ind_1({\cal D})$, then
$(E_{ij},0) \in {\cal
D}({\mathbb C})$, and
so we can put
$X_{ij}=0$. Analogously, if $(i',j')\in \Ind_2({\cal D})$, then
$(0,E_{i'j'}) \in {\cal
D}({\mathbb C})$, and
so we can put
$X_{i'j'}=0$.
Denote
\begin{equation}\label{gamma}
\gamma:=
\sum_{(i,j)\notin \Ind_1({\cal D})}
\|X_{ij}\|+ \sum_{(i',j')\notin{ \Ind_2({\cal D})}} \|X^{'}_{i'j'}\|,
\end{equation}
where $ \| \cdot \|$ denotes the Frobenius norm. 

\begin{theorem}[Upper bound for the norm of miniversal deformations]
\label{defteo}
Let $(A,B)\in({\mathbb
C}^{\,\hat{n}\times \hat{n}}_s,{\mathbb
C}^{\,\hat{n}\times \hat{n}}_s)$ and
let $\varepsilon \in \mathbb{R}$ such that $0< \varepsilon < \kappa$
{where $\kappa = (\max \{ 1+\gamma(\alpha+1)(2+ \gamma),1+\gamma(\beta+1)(2+ \gamma)\})^{-1}$ with $\alpha:=\|A\| , \beta:=\|B\|$  and $\gamma$ is defined in \eqref{gamma}}.
For each pair of symmetric ${\hat{n}}$-by-${\hat{n}}$ matrices $(M,N)$ satisfying
\begin{equation} \label{15}
\|M\|<\varepsilon^{2},\qquad
\|N\|<\varepsilon^{2},
\end{equation}
there exists a matrix $S=I_{\hat{n}}+X$
depending holomorphically on the
entries of $(M,N)$ in a neighborhood of
zero such that
\[S^{T}(A+M,B+N)S=(A+P,B+Q), \ \ (P,Q) \in {\cal D}({\mathbb C}), \ \|  P \|<\varepsilon, \text{and} \ \| Q \|<\varepsilon, \] 
where ${\mathbb C}^{\,\hat{n}\times \hat{n}}_{s}\times{\mathbb C}^{\,\hat{n}\times \hat{n}}_{s}=T_{(A,B)_{\text{\rm can}}} + {\cal D}({\mathbb C})$.
\end{theorem}
\noindent The proof of Theorem \ref{defteo} is analogous to the proof of Theorem~2.2 in \cite{Dmyt16} and  we omit it. 

In the following corollary we compute explicitly the codimensions of orbits of pairs of
symmetric complex matrices. Alternatively, these codimensions can be computed using Corollary~2.2 and Theorem~2.3 of \cite{DmKS14}. 
\begin{corollary}\label{soldim}
The codimension of congruence orbit of a pair of
symmetric complex matrices in the congruence canonical form $(A,B)_{\rm can}=
\bigoplus_{i=1}^{a} \mathcal  H_{h_i}(\lambda_i)
 \oplus
\bigoplus_{j=1}^{b} \mathcal K_{k_j}
 \oplus
\bigoplus_{r=1}^{c} \mathcal L_{l_r}$ 
can be computed as follows:
\begin{equation}\label{dsj}
\codim {(A,B)_{\rm can}} = c_{\mathcal  H} + c_{\mathcal  K} +
c_{\mathcal  L} + c_{\mathcal  H\mathcal  H} + c_{\mathcal  K\mathcal  K} + c_{\mathcal  L\mathcal  L} + c_{\mathcal  H\mathcal  K}
+ c_{\mathcal  H\mathcal  L} + c_{\mathcal  K\mathcal  L},
\end{equation}
where the summands correspond to
\begin{itemize}
\item the direct summands of
    \eqref{kus} of the same type:
$$ c_{\mathcal  H}:= \sum_{\begin{smallmatrix} i =1 \end{smallmatrix}}^{a} h_i,\qquad
c_{\mathcal  K}:= \sum_{\begin{smallmatrix} j =1 \end{smallmatrix}}^{b} k_j,
\qquad
c_{\mathcal  L}:= 2\sum_{\begin{smallmatrix} r =1 \end{smallmatrix}}^{c} (l_r+1); $$

\item the pairs of direct summands
    of \eqref{kus} of the same
    type:
$$ c_{\mathcal  H\mathcal  H}:= \sum_{\begin{smallmatrix} i < i' \\
\lambda_i=\lambda_{i'} \end{smallmatrix}}
\min (h_i, h_{i'}), \qquad c_{\mathcal  K\mathcal  K}:= \sum_{j < j'} \min (k_j,k_{j'}),$$
$$ c_{\mathcal  L\mathcal  L}:= \sum_{r < r'} \left( 2 \max (l_r,l_{r'}) + \varepsilon_{rr'} \right),
\quad \text{ in which } \varepsilon_{rr'} := \begin{cases} 2 & \text{if } l_r=l_{r'}, \\
1 & \text{if } l_r \neq l_{r'}; \end{cases} $$
\item the pairs of direct summands
    of \eqref{kus} of different
    types:
$$ c_{\mathcal  H\mathcal  K}:= 0, \qquad c_{\mathcal  H\mathcal  L}:= c\sum_{i} h_i,
\qquad c_{\mathcal  K\mathcal  L}:= c\sum_{j} k_j. $$
\end{itemize}
\end{corollary}
\begin{proof}
The numbers $c_{\mathcal  H}, c_{\mathcal  K}, \dots , c_{\mathcal  K\mathcal  L}$ are obtained by counting the independent parameters in the miniversal deformations from Theorem \ref{teo2}.  
\end{proof}

\section{Proof of Theorem \ref{teo2}}
\label{sec3}

\subsection{Construction of miniversal deformations}

To make this paper self-contained, we briefly describe how 
to construct the simplest miniversal
deformations. 
This method is presented with more details in \cite{Dmyt16,DmFS12,DmFS14} and will be used to prove Theorem \ref{teo2}.


For a matrix pair $(A,B)$, the deformation 
\begin{equation}\label{edr}
{\cal U}(\vec
{\varepsilon}):=
\left(A+\sum_{i=1}^{\hat{n}} \sum_{j=i}^{\hat{n}}
\varepsilon_{ij}E_{ij},\
B+\sum_{i=1}^{\hat{n}} \sum_{j=i}^{\hat{n}}
{\varepsilon}^{'}_{ij}E_{ij}\right) 
\end{equation} 
is universal in the following sense: every deformation
${\cal
B}(\mu_1,\dots,\mu_l)$
of $(A,B)$ has the form 
${\cal
U}(\vec{\varphi}
(\mu_1,\dots,\mu_l)),$
where
$\varphi_{ij}(\mu_1,\dots,\mu_l)$
are convergent in a
neighborhood of $\vec
0$ power series such
that
$\varphi_{ij}(\vec 0)=
0$. Hence ``every
deformation ${\cal
B}(\mu_1,\dots,\mu_l)$''
in Definition \ref{d}
can be replaced by
${\cal U}(\vec{\varepsilon})$. 

\hide{
\begin{lemma}\label{lem}
The following two
conditions are
equivalent for any
deformation ${\cal A}(\lambda_1,\dots,\lambda_k)$
of pair of matrices $(A,B)$:
\begin{itemize}
  \item[\rm(i)]
The deformation ${\cal
A}(\lambda_1,\dots,\lambda_k)$
is versal.
  \item[\rm(ii)]
The deformation
\eqref{edr} is
equivalent to ${\cal
A}(\varphi_1(\vec{\varepsilon}),\dots,
\varphi_k(\vec{\varepsilon}))$
in which all
$\varphi_i(\vec{\varepsilon})$
are convergent in a
neighborhood of\/
$\vec 0$ power series
such that
$\varphi_i(\vec 0)=0$.
\end{itemize}
\end{lemma}
}

The following lemma ensures
that any matrix pair
with entries $0$ and
$*$ that satisfies
\eqref{jyr} can be
taken as a versal deformation of $(A,B)$. 
For the proof of Lemma
\ref{t2.1} see \cite[Lemma~3.2]{Dmyt16}.
Recall that, for a subspace $U$ of a vector space $V$, a \emph{coset of \/$U$ in $V$} is a set $v+U$, where $v\in V$.

\begin{lemma}
 \label{t2.1}
Let $(A,B)\in ({\mathbb
C}^{\,\hat{n}\times \hat{n}}_s, {\mathbb
C}^{\,\hat{n}\times \hat{n}}_s)$ and
let $\cal D$ be a pair of
$(0,*)$ matrices of size
$\hat{n}\times \hat{n}$. The
following are
equivalent:
\begin{itemize}
  \item[\rm(i)]
The deformation
$(A,B)+{\cal D}(\vec{\varepsilon})$
defined in \eqref{a2z}
is miniversal.

  \item[\rm(ii)]
The vector space
$({\mathbb
C}^{\,\hat{n}\times \hat{n}}_s, {\mathbb
C}^{\,\hat{n}\times \hat{n}}_s)$
decomposes into the
direct sum
\begin{equation}\label{a4}
({\mathbb C}^{\,\hat{n}\times
\hat{n}}_s, {\mathbb
C}^{\,\hat{n}\times \hat{n}}_s)=T_{(A,B)}
+ {\cal
D}({\mathbb C}), \quad T_{(A,B)} \cap {\cal
D}({\mathbb C}) = \{ (A,B) \}.
\end{equation}

  \item[\rm(iii)]
Each coset of
$T_{(A,B)}$
in $({\mathbb
C}^{\,\hat{n}\times \hat{n}}_s, {\mathbb
C}^{\,\hat{n}\times \hat{n}}_s)$
contains exactly one
matrix of the form
${\cal D}$.
\end{itemize}
\end{lemma}
\hide{
\begin{proof}
Define the action of
the group
$GL_n(\mathbb C)$ of
nonsingular $n$-by-$n$
matrices on the space
$[{\mathbb
C}^{\,n\times n}_s,{\mathbb
C}^{\,n\times n}_s]$ by
\[
(A,B)^S=S^T (A,B)S,\qquad
(A,B)\in [{\mathbb
C}^{\,n\times n}_s,{\mathbb
C}^{\,n\times n}_s],\quad
S\in GL_n(\mathbb C).
\]
The orbit $(A,B)^{GL_n}$
of $(A,B)$ under this
action consists of all pairs
of symmetric
matrices that are
congruent to the pair $(A,B)$.

The space $T_{(A,B)}$
is the tangent space
to the orbit
$(A,B)^{GL_n}$ at the
point $(A,B)$ since
\begin{multline*}
(A,B)^{I+\varepsilon C}=
(I+\varepsilon
C)^T(A,B)(I+\varepsilon C) \\
=(A,B)+ \varepsilon(C^T (A,B)+
(A,B)C) +\varepsilon^2C^T
(A,B)C
\end{multline*}
for all $n$-by-$n$
matrices $C$ and
$\varepsilon\in\mathbb
C$. Hence ${\cal
D}(\vec
{\varepsilon})$ is
transversal to the
orbit $(A,B)^{GL_n}$ at
the point $(A,B)$ if
\[
({\mathbb C}^{\,n\times
n}_s,{\mathbb
C}^{\,n\times n}_s)=T_{(A,B)} +
{\cal D}({\mathbb C})
\]
(see definitions in
\cite[\S\,29E]{Arno97};
two subspaces of a
vector space are
called
\emph{transversal} if
their sum is equal to
the whole space).

This proves the
equivalence of (i) and
(ii) since a
transversal (of the
minimal dimension) to
the orbit is a
(mini)versal
deformation
\cite[Section~1.6]{Arno72}. The
equivalence of (ii)
and (iii) is obvious.
\end{proof}
}

Recall that versality
of each deformation
$(A,B)+{\cal D}(\vec
{\varepsilon})$ in
which ${\cal D}$
satisfies \eqref{a4} means
that there exists a
deformation ${\cal
I}(\vec
{\varepsilon})$ of the
identity matrix such
that ${\cal
D}(\vec{\varepsilon})=
{\cal
I}(\vec{\varepsilon})^{T}
{\cal
U}(\vec{\varepsilon})
{\cal
I}(\vec{\varepsilon})$,
where ${\cal U}(\vec
{\varepsilon})$ is
defined in
\eqref{edr}.

\hide{
Thus, a simplest
miniversal deformation
of $(A,B)\in ({\mathbb
C}^{\,n\times n}_s, {\mathbb
C}^{\,n\times n}_s)$ can
be constructed as
follows. Let
$(T_1,\dots,T_r)$ be a
basis of the space
$T_{(A,B)}$,
and let
$(E_1,\dots,E_{\frac{n(n+1)}{2}};F_1,\dots,F_{\frac{n(n+1)}{2}})$ be
the basis of $({\mathbb
C}^{\,n\times n}_s, {\mathbb
C}^{\,n\times n}_s)$
consisting of all
elementary matrices
$(E_{ij},F_{ij})$. Removing
from the sequence
$(T_1,\dots,T_r,E_1,\dots,E_{\frac{n(n+1)}{2}},F_1,\dots,F_{\frac{n(n+1)}{2}})$
every pair of matrices that is a
linear combination of
the preceding
matrices, we obtain a
new basis \newline $(T_1,\dots,
T_r,
E_{i_1},\dots,E_{i_k},
F_{i_1},\dots,F_{i_m})$
of the space $({\mathbb
C}^{\,n\times n}_s, {\mathbb
C}^{\,n\times n}_s)$. By
Lemma~\ref{t2.1}, the
deformation
\[
{\cal
A}(\varepsilon_1,\dots,
\varepsilon_k,f_1, \dots, f_m)=
(A+\varepsilon_1
E_{i_1}+\dots+\varepsilon_kE_{i_k},B+f_1F_{i_1}+ \dots + f_mF_{i_m})
\]
is miniversal.
}
A simplest
miniversal deformation
of $(A,B)\in ({\mathbb
C}^{\,\hat{n}\times \hat{n}}_s, {\mathbb
C}^{\,\hat{n}\times \hat{n}}_s)$ can
be constructed as
follows. Let $(E_1,\dots,E_{{\hat{n}}({\hat{n}}+1)})$ be
the basis of $({\mathbb
C}^{\,\hat{n}\times \hat{n}}_s, {\mathbb
C}^{\,\hat{n}\times \hat{n}}_s)$
in which every $E_k$ is either of the form 
$(E_{ij},0)$ or $(0,E_{i'j'})$ and let $(T_1,\dots,T_r)$ be a
basis of the space
$T_{(A,B)}$. By removing
from the sequence
$(T_1,\dots, T_r,E_1,\dots,E_{{\hat{n}}({\hat{n}}+1)})$
every pair of matrices that is a
linear combination of
the preceding
matrices, we receive a
new basis $(T_1,\dots, T_r,
E_{i_1},\dots,E_{i_k})$
of the space $({\mathbb
C}^{\,\hat{n}\times \hat{n}}_s, {\mathbb
C}^{\,\hat{n}\times \hat{n}}_s)$. By
Lemma \ref{t2.1}, the
deformation
\begin{align*}
{\cal
A}(\varepsilon_1,\dots,
\varepsilon_{k_1},\varepsilon'_1, \dots, \varepsilon'_{k_2})&=
(A,B)+\varepsilon_1E_{1}+\dots+\varepsilon_{k_1}E_{i_{k_1}} +\varepsilon'_1E_{i_{k_1+1}}+\dots+\varepsilon'_{k_2}E_{i_{k}}\\
&=
(A,B)+\varepsilon_1(E_{i_1,j_1},0)+\dots+\varepsilon_{k_1}(E_{i_{k_1}j_{k_1}},0) \\ 
&+\varepsilon'_1(0,E_{i_{k_1+1},j_{k_1+1}})+\dots+\varepsilon'_{k_2}(0,E_{i_{k},j_{k}}),
\end{align*}
where $k_1+k_2=k$, is miniversal.

For each pair of $m\times m$
symmetric matrices $(A_1,B_1)$ and each pair
$n\times n$ symmetric matrices
$(A_2,B_2)$, define the vector
spaces
\begin{equation}\label{neh}
 V(A_1,B_1):=\{ S^T(A_1,B_1)+(A_1,B_1)S, \text{ where } S\in
{\mathbb C}^{m\times m} \}.
\end{equation}
\begin{equation}  \label{neh1}
\begin{aligned}
 V((A_1,B_1),(A_2,B_2))&:=\{(
R^T(A_2,B_2)+(A_1,B_1)S,
S^T(A_1,B_1)+(A_2,B_2)R), \\
&\text{ where } \ S\in {\mathbb C}^{m\times n} \ \text{ and } \  R\in
 {\mathbb C}^{n\times m} \}.
 \end{aligned}
\end{equation}

\begin{lemma}\label{thekd}
Let
$(A,B)=(A_1,B_1)\oplus\dots\oplus
(A_t, B_t)$ be a
block-diagonal matrix
in which every $(A_i,B_i)$
is $n_i\times n_i$.
Let $\cal D$ be a pair of
$(0,*)$ matrices of the same size as $(A,B)$ and 
partitioned into
blocks $(D_{ij},D'_{ij})$
conformably to the
partition of $(A,B)$, see \eqref{grsd}. Then
$(A,B)+{\cal D}(E,E')$ is a
simplest miniversal
deformation of $(A,B)$ for
congruence if and only if
\begin{itemize}
  \item[\rm(i)]
every coset of\/
$V(A_i,B_i)$ in $({\mathbb
C}^{n_i\times n_i}_s, {\mathbb
C}^{n_i\times n_i}_s)$
contains exactly one
matrix of the form
$(D_{ii},D'_{ii})$, and

  \item[\rm(ii)]
every coset of
$V((A_i,B_i),(A_i^T,B_i^T))$ in
$({\mathbb
C}^{n_i\times
n_j}, {\mathbb
C}^{n_i\times
n_j}) \oplus  ({\mathbb
C}^{n_j\times n_i}, {\mathbb
C}^{n_j\times n_i})$
contains exactly two pairs of
matrices $((W_1,W_2),(W_1^T,W_2^T))$, where $(W_1,W_2)$ is of the
form $(D_{ij},D'_{ij})$ and $(W_1^T,W_2^T)$ is of the form
$(D_{ji},D'_{ji})=(D_{ij}^T,D_{ij}^{'T})$.
\end{itemize}
\end{lemma}\noindent See \cite[Lemma~3.3]{Dmyt16} for the proof of Lemma \ref{thekd}.

\begin{corollary}\label{the}
By Lemma \ref{thekd},
$(A,B)+{\cal D}(\vec
{\varepsilon})$ is a
miniversal deformation
of $(A,B)$ if and only if
each submatrix of the
form
\begin{equation*}\label{a8}
\left(
\begin{bmatrix}
  A_i+D_{ii}(\vec
{\varepsilon}) &
  D_{ij}(\vec
{\varepsilon})\\
  D_{ji}(\vec
{\varepsilon}) &A_j+
 D_{jj}(\vec
{\varepsilon})
\end{bmatrix}, 
\begin{bmatrix}
  B_i+D'_{ii}(\vec
{\varepsilon}) &
  D'_{ij}(\vec
{\varepsilon})\\
  D'_{ji}(\vec
{\varepsilon}) &B_j+
D'_{jj}(\vec
{\varepsilon})
\end{bmatrix}
\right),\qquad i<j
\end{equation*}
is a miniversal
deformation of the pair
$(A_i\oplus A_j,B_i\oplus B_j)$.
\end{corollary}

To prove Theorem \ref{teo2}, we show that the pairs
\eqref{Hdef}--\eqref{ktlm}
satisfy the conditions
(i) and (ii) of Lemma \ref{thekd}.  
Each ${\cal X}_i$ in \eqref{gto} is of the
form ${\cal H}_n(\lambda)$,
or ${\cal K}_n$, or
${\cal L}_n$, and so there
are 9 types of pairs
${\cal D}({\cal X}_i)$ and
${\cal D}({\cal X}_i,{\cal X}_j)$
with $i<j$; they are given in 
\eqref{Hdef}--\eqref{ktlm}.

\subsection{Diagonal blocks
of matrices of $\cal
D$}

In Sections \ref{dhndkn} and \ref{dln} we verify the
condition (i) of Lemma
\ref{thekd} for the diagonal blocks of $\cal D$ defined in
part (i) of Theorem~\ref{teo2}. 

\subsubsection{Diagonal blocks
${\cal D}({\cal H}_{n}(\lambda))$ and ${\cal D}({\cal K}_{n})$}
\label{dhndkn}

We start by considering the pair of blocks ${\cal H}_{n}(\lambda)$.
Due to Lemma
\ref{thekd}(i), it
suffices to prove that
each pair of symmetric
$n$-by-$n$ matrices
$(A,B)$ can be reduced to
exactly one pair of matrices of
the form \eqref{Hdef} by adding
\begin{align*}
\delta (A,B)= (\delta A, \delta B) &=S^T (\Delta_n , \Lambda_n(\lambda))
+(\Delta_n , \Lambda_n(\lambda))S \\
&=(S^T\Delta_n + \Delta_n S ,S^T \Lambda_n(\lambda) +
\Lambda_n(\lambda)S)
\end{align*}
in which
$S$ is an arbitrary
$n$-by-$n$ matrix. Obviously, that adding $S^T\Delta_n + \Delta_n S$ we reduce $A$ to zero.
To preserve $A$, we must hereafter take $S$ such that $S^T\Delta_n + \Delta_n S=0$.
This means that $S$ is skew symmetric with respect to its anti-diagonal.
Therefore, we reduce $B$ by adding
\begin{multline}
\delta B=S^T \Lambda_n(\lambda) + \Lambda_n(\lambda) S\\
{\footnotesize
=\begin{bmatrix}
s_{11}& s_{12} & \ldots &s_{1,n-2}&s_{1,n-1}& 0\\
s_{21}& s_{22} & \ldots &s_{2,n-2}&0& -s_{1,n-1}\\
s_{31}& s_{32} & \ldots &0&-s_{2,n-2}& -s_{1,n-2}\\
\vdots&\vdots&\ddd&\vdots& \vdots & \vdots\\
s_{n-1,1}& 0 & \ldots & -s_{32}& -s_{22}& -s_{12}\\
0& -s_{n-1,1} & \ldots & -s_{31}& -s_{21}& -s_{11}\\
\end{bmatrix}
\begin{bmatrix}
&&&&&\lambda\\
&&&&\lambda&1\\
&&&\lambda&1&\\
&&\ddd&\ddd&&\\
&\lambda&1&&&\\
\lambda&1&&&&\\
\end{bmatrix}} \\
{\footnotesize
+\begin{bmatrix}
&&&&&\lambda\\
&&&&\lambda&1\\
&&&\lambda&1&\\
&&\ddd&\ddd&&\\
&\lambda&1&&&\\
\lambda&1&&&&\\
\end{bmatrix}
\begin{bmatrix}
s_{11}& s_{21} & \ldots &s_{n-2,1}&s_{n-1,1}& 0\\
s_{12}& s_{22} & \ldots &s_{n-2,2}&0& -s_{n-1,1}\\
s_{13}& s_{23} & \ldots &0&-s_{n-2,2}& -s_{n-2,1}\\
\vdots&\vdots&\ddd&\vdots& \vdots & \vdots\\
s_{1,n-1}& 0 & \ldots & -s_{23}& -s_{22}& -s_{21}\\
0& -s_{1,n-1} & \ldots & -s_{13}& -s_{12}& -s_{11}\\
\end{bmatrix}}\\ \label{34hd}
{\footnotesize
=\begin{bmatrix}
0& 0& s_{1,n-1}& s_{1,n-2}& \ldots& s_{13}& s_{12}\\
0& -2s_{1,n-1}& -s_{1,n-2}& s_{2,n-2}-s_{1,n-3}& \ldots& s_{23}-s_{12}& s_{22}-s_{11}\\
s_{1,n-1}& -s_{1,n-2}& -2s_{2,n-2}& -s_{2,n-3}&  \ldots& s_{33}-s_{22}& s_{32}-s_{21}\\
s_{1,n-2}& s_{2,n-2}-s_{1,n-3}& -s_{2,n-3}& -2s_{3,n-3}& \ldots & s_{43}-s_{32}& s_{42}-s_{31}\\
\vdots&\vdots&\vdots&\vdots& \ddots&\vdots & \vdots\\
s_{13}& s_{23}-s_{12}& s_{33}-s_{22}& s_{43}-s_{32}& \ldots& -2s_{n-2,2}& -s_{n-2,1}\\
s_{12}& s_{22}-s_{11}& s_{32}-s_{21}& s_{42}-s_{31}& \ldots& -s_{n-2,1}& -2s_{n-1,1}
\end{bmatrix}
}
\end{multline}\noindent The upper part of each anti-diagonal of $\delta B$ has unique variables. Thus adding $\delta B$ we reduce each anti-diagonal of $B$ independently. We start from the upper left hand corner for each of the first $n$ anti-diagonals (note that two first anti-diagonals consist just of zeros); we have the following system of equations: 
\begin{equation} \label{firstm} 
\small{
\left[
 \begin{matrix}
1&&&\\
-1&1&&\\
&\ddots&\ddots&\\
&&-1&1\\
&&&-\alpha
\end{matrix}
\right]
\left[
 \begin{matrix}
s_{-,-} \\
s_{-,-} \\
\vdots \\
s_{-,-}
\end{matrix}
\right]
=
\left[
 \begin{matrix}
b_1 \\
b_2 \\
\vdots \\
b_{t}
\end{matrix}
\right], \quad
}
\end{equation}
where $\alpha = 2$ for odd anti-diagonals, $\alpha = 1$ for even anti-diagonals,  respectively, each $s_{-,-}$ denotes the corresponding entry of \eqref{34hd} (we skip writing the subindices explicitly and only write ``${-,-}$'', it will allow us to refer to \eqref{firstm} in following sections), and each $b_i$ denotes the corresponding entry of $B$. The matrix of the system \eqref{firstm} has $t-1$ columns and $t$ rows.
By the Kronecker--Capelli theorem, 
the system \eqref{firstm} does not have a solution, since the rank of \eqref{firstm} is equal to $(t-1)$ but the rank of the extended matrix of the system is $t$. Nevertheless, if we turn down the first or the last equation of the system
(i.e. we do not set the first or the last element of the corresponding anti-diagonal of $B_{21}$ to zero), then \eqref{firstm} will have a solution. We chose to turn down the last equation to set to zero more elements (on odd anti-diagonals). 
For the remaining $(n-1)$ anti-diagonals we have a system of equations like \eqref{firstm} without the first equation, which has a solution. 
Therefore we can reduce the matrix $B$ to \eqref{Hdef}. 

The result does not
depend on $\lambda$ therefore $ {\cal D}({\cal H}_n(\lambda))=(0,0^{\nwarrow\!\!\!\! \text{\raisebox{1.5pt}{$\nwarrow$}} \!\!\!\! \text{\raisebox{3pt}{$\nwarrow$}}})$ and ${\cal D}({\cal K}_n)=(0^{\nwarrow\!\!\!\! \text{\raisebox{1.5pt}{$\nwarrow$}} \!\!\!\! \text{\raisebox{3pt}{$\nwarrow$}}},0)$.

\hide{

Notably, ${\cal D}({\cal H}_{n}(\lambda))$ and ${\cal D}({\cal K}_{n})$ are the same up to the permutation of matrices.

which has a solution by the Kronecker-Capelli theorem but for each half of the fist $n$ anti-diagonals we have a system of equations with the matrix

where $x_1 \ldots x_k$ are corresponding elements of $B$.  If we drop the first or the last equation of the system
then it will have a solution. 

\begin{equation*} 
\small{
\left[
 \begin{matrix}
1&&&0\\
-1&1&&\\
&\ddots&\ddots&\\
&&-1&1\\
0&&&-2
\end{matrix}
\right]
\left[
 \begin{matrix}
s_{1,k+1} \\
s_{2k} \\
\vdots \\
s_{tt}
\end{matrix}
\right]
=
\left[
 \begin{matrix}
b_1 \\
b_2 \\
\vdots \\
b_{t-1}
\end{matrix}
\right] \quad \text{for odd anti-diagonals},
}
\end{equation*}
}

\subsubsection{Diagonal blocks ${\cal D}({\cal L}_{n})$}
\label{dln}
Similarly to Section \ref{dhndkn}, i.e. using Lemma \ref{thekd}(i), we prove that
each pair $(A,B)=\bigg( \begin{bmatrix}
A_{11}&A_{12} \\
A_{21}&A_{22}
\end{bmatrix},
\begin{bmatrix}
B_{11}&B_{12} \\
B_{21}&B_{22}
\end{bmatrix}
\bigg)$ of symmetric
$(2n+1)$-by-$(2n+1)$ matrices can be reduced
to \eqref{Ldef}
by adding
{\small 
\begin{multline}\label{moh}
\delta (A,B)=\bigg( \begin{bmatrix}
\delta A_{11}&\delta A_{12} \\
\delta A_{21}&\delta A_{22}
\end{bmatrix},
\begin{bmatrix}
\delta B_{11}&\delta B_{12} \\
\delta B_{21}&\delta B_{22}
\end{bmatrix}
\bigg)\\ =\begin{bmatrix}
S_{11}^T&S_{21}^T
 \\ S_{12}^T&S_{22}^T
\end{bmatrix}
\bigg( \begin{bmatrix}
0&F_n^T \\
F_n&0
\end{bmatrix},
\begin{bmatrix}
0&G_n^T \\
G_n&0
\end{bmatrix}
\bigg) + \bigg( \begin{bmatrix}
0&F_n^T \\
F_n&0
\end{bmatrix},
\begin{bmatrix}
0&G_n^T \\
G_n&0
\end{bmatrix}  \bigg)
\begin{bmatrix}
S_{11}&S_{12}
 \\ S_{21}&S_{22}
\end{bmatrix}
    \\=
\bigg( \begin{bmatrix}
S_{21}^TF_n+F_n^TS_{21}&
S_{11}^TF_n^T+F_n^TS_{22}\\
S_{22}^TF_n+F_nS_{11}&
S_{12}^TF_n^T+F_nS_{12}
\end{bmatrix}, 
\begin{bmatrix}
S_{21}^TG_n+G_n^TS_{21}&
S_{11}^TG_n^T+G_n^TS_{22}\\
S_{22}^TG_n+G_nS_{11}&
S_{12}^TG_n^T+G_nS_{12}
\end{bmatrix} \bigg),
\end{multline}}\noindent where $S=[S_{ij}]_{i,j=1}^2$
is an arbitrary $(2n+1) \times (2n+1)$ matrix.
Notably, each pair of blocks $(A_{ij},B_{ij}), i,j=1,2$ of $(A,B)$ is changed independently. 
The pair of blocks $(A_{11},B_{11})$ is changed by adding 
$(S_{21}^TF_n+F_n^TS_{21},S_{21}^TG_n+G_n^TS_{21})$. 
Adding $\delta A_{11}=S_{21}^TF_n+F_n^TS_{21}$
we reduce each $(n+1)$-by-$(n+1)$ symmetric matrix $A_{11}$ to
 $0_{*}$. To preserve $A_{11}$, we must hereafter take $S_{21}$ such that
 $F_n^TS_{21}=-S_{21}^TF_n$. Therefore 
$$S_{21}=\begin{bmatrix}
0&s_{12}&s_{13}&\ldots & s_{1n}&0 \\
-s_{12}&0&s_{23}&\ldots & s_{2n}&0 \\
-s_{13}&-s_{23}&0&\ldots & s_{3n}&0 \\
\vdots&\vdots&\vdots&\ddots & \vdots &\vdots  \\
-s_{1n}&-s_{2n}&-s_{3n}&\ldots & 0&0 \\
\end{bmatrix},$$
i.e. $S_{21}$ without the last column is skew symmetric.
Now we reduce $B_{11}$ by adding
{\small
\begin{multline*}
\delta B_{11}={\footnotesize \begin{bmatrix}
0&-s_{12}&-s_{13}&\ldots & -s_{1n}\\
s_{12}&0&-s_{23}&\ldots & -s_{2n}\\
s_{13}&s_{23}&0&\ldots & -s_{3n}\\
\vdots&\vdots&\vdots&\ddots & \vdots \\
s_{1n}&s_{2n}&s_{3n}&\ldots & 0\\
0&0&0&\ldots &0
\end{bmatrix} }
\begin{bmatrix}
0&1&&\\
&\ddots&\ddots&\\
&& 0&1\\
\end{bmatrix}
 \\ +
\begin{bmatrix}
0&&\\
1&\ddots &\\
&\ddots & 0\\
& & 1\\
\end{bmatrix}{\footnotesize
\begin{bmatrix}
0&s_{12}&s_{13}&\ldots & s_{1n}&0 \\
-s_{12}&0&s_{23}&\ldots & s_{2n}&0 \\
-s_{13}&-s_{23}&0&\ldots & s_{3n}&0 \\
\vdots&\vdots&\vdots&\ddots & \vdots &\vdots  \\
-s_{1n}&-s_{2n}&-s_{3n}&\ldots & 0&0 \\
\end{bmatrix}}
    \\=
{\footnotesize \begin{bmatrix}
0& 0& -s_{12}& \ldots & & -s_{1,n-1}& -s_{1n} \\
0& 2s_{12}& s_{13}&\ldots & & s_{1n}-s_{2,n-1}& -s_{2n}\\
-s_{12}& s_{13}& 2s_{23}& \ddots && s_{2n}-s_{3,n-1}& -s_{3n}\\
\vdots&\vdots&\ddots&\ddots & \ddots & \vdots& \vdots\\
-s_{1,n-2}& s_{1,n-1}-s_{2,n-2}& &\ddots & 2s_{n-2,n-1}& s_{n-2,n}& -s_{n-1,n}\\
-s_{1,n-1}& s_{1n}-s_{2,n-1}& &\ldots & s_{n-2,n}& 2s_{n-1,n}& 0\\
-s_{1n}& -s_{2n}& &\ldots & -s_{n-1,n}& 0& 0
\end{bmatrix}}.
\end{multline*}}\noindent Similarly to \eqref{34hd}, we reduce $B_{11}$ anti-diagonal-wise and the system of equations corresponding to each anti-diagonal is \eqref{firstm}. 
Therefore we reduce $B_{11}$ to the form $0^{\nwsearrow\!\!\!\! \text{\raisebox{1.5pt}{$\nwsearrow$}} \!\!\!\! \text{\raisebox{3pt}{$\nwsearrow$}}}$.


The pair of blocks $(A_{21},B_{21})$ is reduced by adding
$\delta (A_{21},B_{21}) = (S_{22}^TF_n+F_nS_{11},S_{22}^TG_n+G_nS_{11}),$
where $S_{11}$ and $S_{22}$ are arbitrary matrices of the corresponding size. Obviously, that adding $S_{22}^TF_n+F_nS_{11}$ we reduce $A_{21}$ to zero. To preserve $A_{21}$, we must hereafter take $S_{11}$ and $S_{22}$ such that
$F_nS_{11}=-S_{22}^TF_n$. Thus 
$$S_{11}=\begin{bmatrix}
&&& 0\\
&-S^T_{22}&& 0\\
&&& \vdots\\
&&& 0\\
-y_{1}&-y_{2}&\ldots & -y_{n+1}\\
\end{bmatrix}$$
and we reduce $B_{12}$ by adding
{\small
\begin{multline*}
\delta B_{12}=S_{22}^TG_n+G_nS_{11}=
{\footnotesize \begin{bmatrix}
s_{11}&s_{12}&s_{13}&\ldots & s_{1n}\\
s_{21}&s_{22}&s_{23}&\ldots & s_{2n}\\
s_{31}&s_{32}&s_{33}&\ldots & s_{3n}\\
\hdotsfor{5}\\
s_{n1}&s_{n2}&s_{n3}&\ldots & s_{nn}\\
\end{bmatrix}}
\begin{bmatrix}
0&1&&\\
&\ddots&\ddots&\\
&& 0&1\\
\end{bmatrix}
 \\ +
\begin{bmatrix}
0&1&&\\
&\ddots&\ddots&\\
&& 0&1\\
\end{bmatrix}
{\footnotesize
\begin{bmatrix}
-s_{11}&-s_{12}&-s_{13}&\ldots & -s_{1n}& 0\\
-s_{21}&-s_{22}&-s_{23}&\ldots & -s_{2n}& 0\\
-s_{31}&-s_{32}&-s_{33}&\ldots & -s_{3n}& 0\\
\hdotsfor{6}\\
-s_{n1}&-s_{n2}&-s_{n3}&\ldots & -s_{nn}&0\\
y_{1}&y_{2}&y_{3}&\ldots &y_{n} & y_{n+1}\\
\end{bmatrix}}
    \\=
    {\footnotesize
\begin{bmatrix}
-s_{21}& s_{11}-s_{22}& s_{12}-s_{23}&  \ldots & s_{1,n-1}-s_{2n}& s_{1n} \\
-s_{31}& s_{21}-s_{32}& s_{22}-s_{33}&  \ldots & s_{2,n-1}-s_{3n}& s_{2n} \\
-s_{41}& s_{31}-s_{42}& s_{32}-s_{43}&  \ldots & s_{3,n-1}-s_{4n}& s_{3n} \\
\hdotsfor{6}\\
-s_{n1}& s_{n-1,1}-s_{n2}& s_{n-1,2}-s_{n3}& \ldots & s_{n,n-1}-s_{nn}& s_{n-1,n} \\
y_1& s_{n1}+y_2 & s_{n2}+y_3 &  \ldots & s_{n,n+1}+y_n & s_{nn}+y_{n+1} \\
\end{bmatrix}}.
\end{multline*}}
Clearly, we can set $B_{12}$ to zero by adding $\delta B_{12}$
(diagonal-wise). 

The reduction of $(A_{21},B_{21})$ follows from the above, since $(A_{12}^T,B_{12}^T)= (A_{21},B_{21})$ and $(\delta A_{12}^T, \delta B_{12}^T)= (\delta A_{21}, \delta B_{21})$.

To the pair of blocks $(A_{22},B_{22})$ we can add
$\delta (A_{22},B_{22})=(S_{12}^TF_n^T+F_nS_{12},S_{12}^TG_n^T+G_nS_{12})$
in which $S_{12}$ is an arbitrary $(n+1)$-by-$n$
matrix. Obviously, that adding $S_{21}^TF_n^T+F_nS_{21}$
we reduce each $n$-by-$n$ symmetric matrix $A_{22}$ to
zero. To preserve $A_{22}$, we must hereafter take $S_{21}$ such
that  $F_nS_{12}=-S_{12}^TF_n^T$. Therefore 
$$S_{12}=\begin{bmatrix}
0&s_{12}&s_{13}&\ldots & s_{1n}& \\
-s_{12}&0&s_{23}&\ldots & s_{2n}& \\
-s_{13}&-s_{23}&0&\ldots & s_{3n}& \\
\vdots&\vdots&\vdots&\ddots & \vdots & \\
-s_{1n}&-s_{2n}&-s_{3n}&\ldots & 0& \\
s_{1,n+1}&s_{2,n+1}&s_{3,n+1}&\ldots & s_{n,n+1}& \\
\end{bmatrix},$$
i.e. $S_{12}$ without the last row is skew symmetric.
Now we reduce $B_{22}$ by adding
{\small
\begin{multline*}
\delta B_{22}= 
{\footnotesize
\begin{bmatrix}
0&-s_{12}&-s_{13}&\ldots & -s_{1n}& s_{1,n+1}\\
s_{12}&0&-s_{23}&\ldots & -s_{2n}& s_{2,n+1}\\
s_{13}&s_{23}&0&\ldots & -s_{3n}& s_{3,n+1}\\
\vdots&\vdots&\vdots&\ddots & \vdots & \vdots\\
s_{1n}&s_{2n}&s_{3n}&\ldots & 0& s_{n,n+1}\\
\end{bmatrix}}
\begin{bmatrix}
0&&\\
1&\ddots &\\
&\ddots & 0\\
& & 1\\
\end{bmatrix}
 \\ +
\begin{bmatrix}
0&1&&\\
&\ddots&\ddots&\\
&& 0&1\\
\end{bmatrix}
{\footnotesize
\begin{bmatrix}
0&s_{12}&s_{13}&\ldots & s_{1n} \\
-s_{12}&0&s_{23}&\ldots & s_{2n} \\
-s_{13}&-s_{23}&0&\ldots & s_{3n} \\
\vdots&\vdots&\vdots&\ddots & \vdots  \\
-s_{1n}&-s_{2n}&-s_{3n}&\ldots & 0 \\
s_{1,n+1}&s_{2,n+1}&s_{3,n+1}&\ldots & s_{n,n+1} \\
\end{bmatrix}}
    \\=
    {\footnotesize
 \begin{bmatrix}
-2s_{12}& -s_{13}&  \ldots& -s_{1n}+s_{2,n-1}& s_{1,n+1}+s_{2n}\\
-s_{13}& -2s_{23}&  \ldots& -s_{2n}+s_{3,n-1}& s_{2,n+1}+s_{3n}\\
\vdots&\vdots&\ddots&\vdots & \vdots  \\
-s_{1n}+s_{2,n-1}& -s_{2n}+s_{3,n-1}& \ldots& -2s_{n-1,n}& s_{n-1,n+1}\\
s_{1,n+1}+s_{2n}& s_{2,n+1}+s_{3n}& \ldots & s_{n-1,n+1}& 2s_{n,n+1}
\end{bmatrix}}. 
\end{multline*}}\noindent Since each anti-diagonal of $\delta B_{22}$ has unique variables, we reduce $B_{22}$ anti-diagonal-wise. For each half of any
anti-diagonal we have the system of equations (\eqref{firstm} without the first equation),
which has a solution. Therefore we reduce each anti-diagonal of $(A_{22},B_{22})$ to zero and so we reduce $(A_{22},B_{22})$ to zero.

Hence ${\cal D}({\cal L}_m)$ is equal to \eqref{Ldef}.

\subsection{Off-diagonal
blocks of matrices of
$\cal D$ that
correspond to summands
of
$(A,B)_{\text{can}}$
of the same type}

In this section we check the
condition (ii) of
Lemma \ref{thekd} for
off-diagonal blocks of
$\cal D$ defined in
Theorem
\ref{teo2}(ii). The horizontal and
vertical strips of these diagonal blocks contain summands of
$(A,B)_{\text{can}}$ of the same type.

\subsubsection{Pairs of
blocks ${\cal
D}({\cal H}_n(\mu),\,
{\cal H}_m(\lambda))$ and
${\cal D}({\cal K}_n,{\cal K}_m)$}
\label{sub4}

Due to Lemma
\ref{thekd}(ii), it
suffices to prove that
each group of four matrices $((A,B),(A^T,B^T))$
can be reduced to
exactly one group of
the form \eqref{lsiu1} (or, respectively \eqref{lsiu2})
by adding
\[
(R^T{\cal H}_m(\lambda)
+{\cal H}_n(\mu)S, S^T
{\cal H}_n(\mu)+ {\cal H}_m(\lambda)
R),\quad S\in
 {\mathbb C}^{n\times m}, R \in
 {\mathbb C}^{m\times n}.
\]
Obviously, if we reduce $(A,B)$ then $(A^T,B^T)$ is reduced automatically. 
The matrix pair $(A,B)$ is reduced by adding 
$$
\delta (A,B) = R^T{\cal H}_m(\lambda)
+{\cal H}_n(\mu)S =(R^T \Delta_m + \Delta_n S,
R^T \Lambda_m(\lambda)+ \Lambda_n(\mu) S).
$$
It is clear that we can set $A$ to zero. To preserve $A$, we must
hereafter take $R$ and $S$ such that
\[
R^T \Delta_m + \Delta_n S=0 \Leftrightarrow R^T =-\Delta_n S\Delta_m.
\]
It follows that $B$ is reduced by adding
\begin{multline*}
\delta B =R^T \Lambda_m(\lambda)+ \Lambda_n(\mu) S=-\Delta_n S\Delta_m\Lambda_m(\lambda)+ \Lambda_n(\mu) S \\ =
\begin{cases}
(\lambda-\mu)s_{n-i+1,j}-s_{n-i+1,j-1}+s_{n-i+2,j}&  \text{if} \ \ \  2 \leq i \leq n , \ \ 2 \leq j \leq m, \\
(\lambda-\mu)s_{n-i+1,j}-s_{n-i+1,j-1}&  \text{if} \ \ \  2 \leq j \leq m, \ \ i=1, \\
(\lambda-\mu)s_{n-i+1,j}+s_{n-i+2,j}&  \text{if} \ \ \ 2 \leq i \leq n, \ \ j=1, \\
(\lambda-\mu)s_{n1}& \text{if} \ \ \  i=j=1.
\end{cases}
\end{multline*}
We have the system of $nm$ equations that has a solution if $\lambda \neq \mu$.
Hence in the case $\lambda \neq \mu$ we can set any
pair $(A,B)$ of $n$-by-$m$ matrices to zero.

For the case $\lambda = \mu$ we have 
\begin{multline*}
\delta B = R^T \Lambda_m(\lambda)+ \Lambda_n(\lambda) S=-\Delta_n S\Delta_m\Lambda_m(\lambda)+ \Lambda_n(\lambda) S \\ =
{\footnotesize
\begin{bmatrix}
0& -s_{n1}& -s_{n2}& -s_{n3}& -s_{n4}& \dots & -s_{n,m-1} \\
s_{n1}& s_{n2}-s_{n-1,1}& s_{n3}-s_{n-1,2}& s_{n4}-s_{n-1,3}& s_{n5}-s_{n-1,4}& \dots & s_{nm}-s_{n-1,m-1}\\
\hdotsfor{7}\\
s_{41}& s_{42}-s_{31}& s_{43}-s_{32}& s_{44}-s_{33}& s_{45}-s_{34}& \dots & s_{4m}-s_{3,m-1} \\
s_{31}& s_{32}-s_{21}& s_{33}-s_{22}& s_{34}-s_{23}& s_{35}-s_{24}& \dots & s_{3m}-s_{2,m-1} \\
s_{21}& s_{22}-s_{11}& s_{23}-s_{12}& s_{24}-s_{13}& s_{25}-s_{14}& \dots & s_{2m}-s_{1,m-1}
\end{bmatrix}}.
\end{multline*}
By adding (anti-diagonal-wise) $\delta B$ we reduce $B$ to the form $0^{\nwarrow}$.

We have shown that ${\cal D}({\cal H}_m(\mu),\, {\cal H}_n(\lambda))$ is equal to \eqref{lsiu1} as well as that ${\cal D}({\cal K}_m,\, {\cal K}_n)$ is equal to \eqref{lsiu2}.
\subsubsection{Pairs of
blocks ${\cal D}({\cal L}_n, {\cal L}_m)$ }
\label{sub42}

Due to Lemma
\ref{thekd}(ii), it
suffices to show that
each four matrices $((A,B),(A^T,B^T))$
can be reduced to
the form \eqref{lsiu3}
by adding
\[
(R^T {\cal L}_m+{\cal L}_nS, S^T {\cal L}_n+{\cal L}_m R),\quad S\in
 {\mathbb C}^{2n+1\times 2m+1},\ R\in
 {\mathbb C}^{2m+1\times 2n+1}.
\]
It is enough to reduce only $(A,B)$ and the pair $(A^T,B^T)$ is reduced automatically.
{\small 
\begin{multline*}
\delta (A,B) = \left( \begin{bmatrix}
\delta A_{11}&\delta A_{12}\\
\delta A_{21}&\delta A_{22}\\
\end{bmatrix},
\begin{bmatrix}
\delta B_{11}&\delta B_{12}\\
\delta B_{21}&\delta B_{22}\\
\end{bmatrix}\right)= R^T{\cal L}_m
+{\cal L}_n S \\ = \left(R^T
\begin{bmatrix}
0&F_m^T\\
F_m&0\\
\end{bmatrix}
+
\begin{bmatrix}
0&F_n^T\\
F_n&0\\
\end{bmatrix}
S,R^T
\begin{bmatrix}
0&G_m^T\\
G_m&0\\
\end{bmatrix}
+
\begin{bmatrix}
0&G_n^T\\
G_n&0\\
\end{bmatrix}
S \right)
\\ =  \bigg( \begin{bmatrix}
R_{12}^TF_m+F^T_nS_{21}&R_{11}^TF^T_m+F^T_nS_{22}\\
R_{22}^TF_m+F_nS_{11}&R_{21}^TF^T_m+F_nS_{12}\\
\end{bmatrix},  
\begin{bmatrix}
R_{12}^TG_m+G^T_nS_{21}&R_{11}^TG^T_m+G^T_nS_{22}\\
R_{22}^TG_m+G_nS_{11}&R_{21}^TG^T_m+G_nS_{12}\\
\end{bmatrix} \bigg).
\end{multline*}}\noindent First we reduce the pair $(A_{11},B_{11})$. Easy to see that by adding
$\delta A_{11}$ we can reduce $A_{11}$ to $0_{*}$.
To preserve $A_{11}$, we must hereafter take $R_{12}$ and $S_{21}$ such that
$R^T_{12}F_m=-F^T_nS_{21}$. This means that 
\begin{equation*}
R^T_{12}=
\begin{bmatrix}
&-Q&\\
0& \ldots &0\\
\end{bmatrix} \text{ and }
S_{21}=
\begin{bmatrix}
&&0\\
Q&& \vdots\\
&&0\\
\end{bmatrix}, \text{where $Q$ is any $n$-by-$m$ matrix.}
\end{equation*}
Hence
{\small 
\begin{multline*}
\delta B_{11} = R^T_{12}G_m+G^T_nS_{21}=
\begin{bmatrix}
&-Q&\\
0& \ldots &0\\
\end{bmatrix}
G_m+G_n^T
\begin{bmatrix}
&&0\\
Q&& \vdots\\
&&0\\
\end{bmatrix} =\\
{\footnotesize
\begin{bmatrix}
0& -q_{11}& -q_{12}& -q_{13}& \ldots & -q_{1,m-1}& -q_{1m} \\
q_{11}& q_{12}-q_{21}& q_{13}-q_{22}& q_{14}-q_{23}& \ldots& q_{1m}-q_{2,m-1}& -q_{2m} \\
q_{21}& q_{22}-q_{31}& q_{23}-q_{32}& q_{24}-q_{33}& \ldots& q_{2m}-q_{3,m-1}& -q_{3m} \\
q_{31}& q_{32}-q_{41}& q_{33}-q_{42}& q_{34}-q_{43}& \ldots& q_{3m}-q_{4,m-1}& -q_{4m} \\
\hdotsfor{7}\\
q_{n-1,1}& q_{n-1,2}-q_{n1}& q_{n-1,3}-q_{n2}& q_{n-1,4}-q_{n3}& \ldots& q_{n-1,m}-q_{n,m-1}& -q_{nm} \\
q_{n1}& q_{n2}& q_{n3}& q_{n4} & \ldots& q_{nm}& 0
\end{bmatrix}}.
\end{multline*}}\noindent
By adding $\delta B_{11}$ we can set each element of
$B_{11}$ to zero except either the first column and the
last row or the first row and the last column.

Now we consider $(A_{12},B_{12})$. We can set $A_{12}$ to zero by
adding $\delta A_{12}$.
To preserve $A_{12}$, we must hereafter take $R_{11}$ and $S_{22}$
such that $R^T_{11}F_m^T=-F_n^TS_{22}$. Thus
\begin{equation*}
R^T_{11}=
\begin{bmatrix}
&&&h_1\\
&-S_{22}&& \vdots\\
0&\ldots&0&h_{n+1}\\
\end{bmatrix},
\end{equation*}
where $S_{22}$ is any $n$-by-$m$ matrix.
Therefore
\begin{multline*}
\delta B_{12}= R^T_{11}G^T_m+G^T_nS_{22}=
\begin{bmatrix}
&&&h_1\\
&-S_{22}&& \vdots\\
0&\ldots&0&h_{n+1}\\
\end{bmatrix}
G_m^T+G_n^TS_{22} =\\
{\footnotesize
\begin{bmatrix}
-s_{12}& -s_{13}& -s_{14}& \ldots& -s_{1m}& h_{1} \\
s_{11}-s_{22}& s_{12}-s_{23}& s_{13}-s_{24}& \ldots& s_{1,m-1}-s_{2m}& h_{2}+s_{1m}\\
\hdotsfor{6}\\
s_{n-1,1}-s_{n2}& s_{n-1,2}-s_{n3}& s_{n-1,3}-s_{n4}& \ldots& s_{n-1,m-1}-s_{nm}& h_{n}+s_{n-1,m}\\
s_{n1}& s_{n2}& s_{n3}& \ldots& s_{n,m-1}& h_{n+1}+s_{nm}
\end{bmatrix}}.
\end{multline*}
If $n \geq m-1$ then we can set $B_{12}$ to zero
by adding $\delta B_{12}$. If $n < m-1$ then we cannot
set the block $B_{12}$ to zero. Then we start the diagonal-wise reduction from the down left hand corner and set the first $(n-1)$ diagonals of $B_{12}$ to zero. We set the next $(m-n)$
diagonals of $B_{12}$ to zero, except the last element of each of them.
The remaining $n$ diagonals
we can set to zero too. Hence we reduce
this pair of blocks $(A_{12},B_{12})$ to the form $(0,0^{\boxminus}_{n+1,m})$.

Now we reduce $(A_{21},B_{21})$. We can set $A_{21}$ to zero
by adding $\delta A_{21}$.
To preserve $A_{21}$, we must hereafter take $R_{22}$ and $S_{11}$
such that $R^T_{22}F_m=-F_nS_{11}$, i.e. 
\begin{equation*}
S_{11}=
\begin{bmatrix}
&&&0\\
&-R^T_{22}&& \vdots\\
&&&0\\
h_1&\ldots&&h_{m+1}\\
\end{bmatrix},
\end{equation*}
where $R^T_{22}$ is any $n$-by-$m$ matrix.
Therefore
{\small 
\begin{multline*}
\delta B_{21}= R^T_{22}G_m+G_nS_{11}=
R^T_{22}G_m+G_n
\begin{bmatrix}
&&&0\\
&-R^T_{22}&& \vdots\\
&&&0\\
h_1&\ldots&&h_{m+1}\\
\end{bmatrix} \\ =
{\footnotesize
\begin{bmatrix}
-r_{21}& r_{11}-r_{22}& r_{12}-r_{23}&\ldots & r_{1,m-1}-r_{2m}& r_{1m}\\
-r_{31}& r_{21}-r_{32}& r_{22}-r_{33}&\ldots & r_{2,m-1}-r_{3m}& r_{2m}\\
-r_{41}& r_{31}-r_{42}& r_{32}-r_{43}&\ldots & r_{3,m-1}-r_{4m}& r_{3m}\\
\hdotsfor{6}\\
-r_{n1}& r_{n-1,1}-r_{n2}& r_{n-1,2}-r_{n3}&\ldots & r_{n-1,m-1}-r_{nm}& r_{n-1,m}\\
h_{1}& r_{n1}+h_{2}& r_{n2}+b_{3}&\ldots & r_{n,m-1}+h_{m}& r_{nm}+h_{m+1}\\
\end{bmatrix}}.
\end{multline*}}\noindent If $m+1 \geq n$ then we can set $B_{21}$ to zero
by adding $\delta B_{21}$. If $m+1 < n$ then we cannot
set the whole $B_{21}$ to zero. By adding $\delta (A_{21},B_{21})$ and arguing as in the previous case we reduce $(A_{21},B_{21})$ to $(0,0^{\boxminus T}_{m+1,n})$.

Now let us consider the pair $(A_{22},B_{22})$. We can set $A_{22}$ to zero
by adding $\delta A_{22}$. To preserve $A_{22}$, we must hereafter
take $R_{21}$ and $S_{12}$ such that $R^T_{21}F_m^T=-F_nS_{12}$, i.e. 
\begin{equation*}
S_{12}=
\begin{bmatrix}
&-Q&\\
g_{1}& \ldots &g_{m}\\
\end{bmatrix} \text{ and }
R^T_{21}=
\begin{bmatrix}
&&h_1\\
Q&& \vdots\\
&&h_n\\
\end{bmatrix}, \text{where $Q$ is any $n$-by-$m$ matrix.}
\end{equation*}
It follows that
{\small 
\begin{multline*}
\delta B_{22}=R^T_{21}G^T_m+G_nS_{12}=
\begin{bmatrix}
&&h_1\\
Q&& \vdots\\
&&h_n\\
\end{bmatrix}
G_m^T+G_n
\begin{bmatrix}
&-Q&\\
g_{1}& \ldots &g_{m}\\
\end{bmatrix} \\
=
{\footnotesize
\begin{bmatrix}
q_{12}-q_{21}&q_{13}-q_{22}&q_{14}-q_{23}&  \ldots & q_{1m}-q_{2,m-1}& h_1-q_{2m}\\
q_{22}-q_{31}&q_{23}-q_{32}&q_{24}-q_{33}&  \ldots & q_{2m}-q_{3,m-1}& h_2-q_{3m}\\
q_{32}-q_{41}&q_{33}-q_{42}&q_{34}-q_{43}&  \ldots & q_{3m}-q_{4,m-1}& h_3-q_{4m}\\
\hdotsfor{6}\\
q_{n-1,2}-q_{n1}&q_{n-1,3}-q_{n2}&q_{n-1,4}-q_{n3}&  \ldots & q_{n-1,m}-q_{n,m-1}& h_{m-1}-q_{nm}\\
q_{n2}+g_1&q_{n3}+g_2&q_{n4}+g_3&  \ldots & q_{nm}-g_{n-1}& g_n+h_m\\
\end{bmatrix}}.
\end{multline*}}
We can set each anti-diagonal of $B_{22}$ to zero independently.
Thus adding $\delta B_{22}$ we reduce $B_{22}$ to zero.

Hence ${\cal D}({\cal L}_m, {\cal L}_n)$ has the form \eqref{lsiu3}.

\subsection{Off-diagonal
blocks of matrices of
$\cal D$ that
correspond to summands
of
$(A,B)_{\text{can}}$
of different types}

At last, we check the
condition (ii) of
Lemma \ref{thekd} for
off-diagonal blocks of
$\cal D$ defined in
Theorem
\ref{teo2}(iii); the
diagonal blocks of
their horizontal and
vertical strips
contain summands of
$(A,B)_{\text{can}}$ of different
types.

\subsubsection{Pairs of
blocks ${\cal D}({\cal H}_n(\lambda), {\cal K}_m)$ }
\label{sub7}

Due to Lemma
\ref{thekd}(ii), it
suffices to prove that
each group of four matrices $((A,B),(A^T,B^T))$
can be reduced to
exactly one group of
the form \eqref{kut}
by adding
\[
(R^T {\cal K}_m+ {\cal H}_n(\lambda) S, S^T {\cal H}_n(\lambda)
+{\cal K}_m R),\quad S\in
 {\mathbb C}^{n\times m},\ R\in
 {\mathbb C}^{m\times n}.
\]
Clearly, we can reduce only $(A,B)$ and the pair $(A^T,B^T)$
is reduced automatically. We have
$$
\delta (A,B) = (\delta A,\delta B) = R^T {\cal K}_m+ {\cal H}_n(\lambda) S = \\ (R^T \Lambda_m(0)+\Delta_nS,
R^T \Delta_m+ \Lambda_n(\lambda)S) .
$$
We can set $A$ to zero by adding $\delta A$. To preserve $A$, we
must hereafter take $R$ and $S$ such that
\[
R^T \Lambda_m(0)+\Delta_nS=0 \Rightarrow S=-\Delta_nR^T \Lambda_m(0).
\]\noindent Thus $B$ is reduced to zero by adding 
$$\delta B = R^T \Delta_m+ \Lambda_n(\lambda)S=
R^T \Delta_m - \Lambda_n(\lambda)\Delta_nR^T \Lambda_m(0).$$
Hence ${\cal D}({\cal H}_n(\lambda), {\cal K}_m)$ is equal to zero.

\subsubsection{Pairs of
blocks ${\cal D}({\cal H}_n(\lambda), {\cal L}_m)$ }
\label{sub8}

Due to Lemma
\ref{thekd}(ii), it
suffices to prove that
each group of four matrices $((A,B),(A^T,B^T))$
can be reduced to
\eqref{hnlm}
by adding
\[
(R^T {\cal L}_m+ {\cal H}_n(\lambda) S,
S^T{\cal H}_n(\lambda) +{\cal L}_mR),\quad S\in
 {\mathbb C}^{n\times 2m+1},\ R\in
 {\mathbb C}^{2m+1\times n}.
\]
Obviously, that we can reduce only $(A,B)$ and
the pair $(A^T,B^T)$ is reduced automatically. We have 
$$
\delta (A,B)=R^T {\cal L}_m+ {\cal H}_n(\lambda) S
=\left(R^T
\begin{bmatrix}
0&F_m^T\\
F_m&0\\
\end{bmatrix}
+
\Delta_n S,R^T
\begin{bmatrix}
0&G_m^T\\
G_m&0\\
\end{bmatrix}
+
\Lambda_n(\lambda) S \right) .
$$
It is clear that we can set $A$ to zero. To preserve $A$, we
must hereafter take $R$ and $S$ such that
\[
R^T
\begin{bmatrix}
0&F_m^T\\
F_m&0\\
\end{bmatrix}
+
\Delta_n S=0
\Rightarrow
S=
-\Delta_n
\begin{bmatrix}
R^T_{11}&R^T_{21}\\
R^T_{12}&R^T_{22}\\
\end{bmatrix}
\begin{bmatrix}
0&F_m^T\\
F_m&0\\
\end{bmatrix}. 
\]
Hence $B$ is reduced by adding 
\begin{multline*}
\delta B =R^T
\begin{bmatrix}
0&G_m^T\\
G_m&0\\
\end{bmatrix}
-
\Lambda_n(\lambda) \Delta_n R^T
\begin{bmatrix}
0&F_m^T\\
F_m&0\\
\end{bmatrix}\\
=
\begin{cases}
-\lambda r_{i,n-1}-r_{i-1,n-1}&  \text{if} \ \ \ 1 \leq j \leq n, \ \ j=1, \\
-\lambda r_{i,m+1+j}-r_{i-1,m+1+j}+r_{i,m+j}&  \text{if} \ \ \  1 \leq i \leq n, \ \ 1 < j < m+1,\\
r_{in}&  \text{if} \ \ \   1 \leq i \leq n, \ \ j=m+1,\\
-\lambda r_{i,j-m-1}-r_{i-1,j-m-1}+r_{i,j-m}&  \text{if} \ \ \  1 \leq i \leq n, \ \ m+1 < j \leq 2m+1,\\
\end{cases}
\end{multline*}
where we put $r_{0t}:=0$. Adding $\delta B$ we reduce $B$ to the form $0^{\leftarrow}$.

Therefore ${\cal D}({\cal H}_n(\lambda), {\cal L}_m)$ is equal to \eqref{hnlm}.

\subsubsection{Pairs of
blocks ${\cal D}({\cal K}_n, {\cal L}_m)$ }
\label{sub9}

Due to Lemma
\ref{thekd}(ii), it
suffices to prove that
each group of four matrices $((A,B),(A^T,B^T))$
can be reduced to
\eqref{ktlm}
by adding
\[
(R^T {\cal L}_m+ {\cal K}_nS,S^T{\cal K}_n
+{\cal L}_mR),\quad S\in
 {\mathbb C}^{n\times 2m+1},\ R\in
 {\mathbb C}^{2m+1\times n}.
\]
As before, we can reduce only $(A,B)$ and
the pair $(A^T,B^T)$ is reduced automatically. We have 
$$
\delta (A,B)=R^T {\cal L}_m+ {\cal K}_n S \\
=\left(R^T
\begin{bmatrix}
0&F_m^T\\
F_m&0\\
\end{bmatrix}
+ \Lambda_n(0) S
,R^T
\begin{bmatrix}
0&G_m^T\\
G_m&0\\
\end{bmatrix}
+
\Delta_n S\right) .
$$
It is easy to check that we can set $B$ to zero. To preserve $B$, we
must hereafter take $R$ and $S$ such that
\[
R^T
\begin{bmatrix}
0&G_m^T\\
G_m&0\\
\end{bmatrix}
+
\Delta_n S=0
\Rightarrow
S=
-\Delta_n
\begin{bmatrix}
R^T_{11}&R^T_{21}\\
R^T_{12}&R^T_{22}\\
\end{bmatrix}
\begin{bmatrix}
0&G_m^T\\
G_m&0\\
\end{bmatrix}.
\]
Thus $A$ is reduced by adding
\begin{multline*}
\delta A = R^T
\begin{bmatrix}
0&F_m^T\\
F_m&0\\
\end{bmatrix}
-
\Lambda_n(0) \Delta_n R^T
\begin{bmatrix}
0&G_m^T\\
G_m&0\\
\end{bmatrix}\\
=
\begin{cases}
r_{in-1}&  \text{if} \ \ \ 1 \leq j \leq n, \ \ j=1,\\
r_{i,m+1+j}-r_{i-1,m+j}&  \text{if} \ \ \  1 \leq i \leq n, \ \ 1 < j < m+1,\\
r_{i-1,n}&  \text{if} \ \ \   1 \leq i \leq n, \ \ j=m+1,\\
r_{i,j-m-1}-r_{i-1,j-m}&  \text{if} \ \ \  1 \leq i \leq n, \ \ m+1 < j \leq 2m+1,\\
\end{cases}
\end{multline*}
where we put $r_{0t}:=0$.

Therefore ${\cal D}({\cal K}_n, {\cal L}_m)$ is equal to \eqref{ktlm}.

\section*{Acknowledgements}
The author is thankful to Vladimir V. Sergeichuk for introducing him to the area of miniversal deformations and, in particular, the problem considered in this paper. The author also thanks to the anonymous referee for the helpful suggestions.

The work was supported by the Swedish Research Council (VR) under grant E0485301, and by eSSENCE, a strategic collaborative e-Science programme funded by the Swedish Research Council.

{\footnotesize
\bibliographystyle{abbrv}
\bibliography{library}
}

\end{document}